\documentclass [12pt]{amsart}
\usepackage[pdftex]{graphicx}

\swapnumbers
\newtheoremstyle{noindent}{4ex}{4ex}{}{0ex}{\sc}{.}{0.5em}{}
\theoremstyle{noindent}
\newtheorem{thm}{Theorem}[section]
\newtheorem{lemma}[thm]{ Lemma}
\newtheorem{prop}[thm]{ Proposition}

\newtheorem{assumption}[thm]{ Assumption}
\newtheorem{defnumb}[thm]{ Definition}
\newtheorem{defn}{Definition}

\newtheoremstyle{noindent2}{0ex}{0ex}{}{0ex}{\sc}{\hspace*{-0.3em}~}{0.5em}{}
\theoremstyle{noindent2}
\newtheorem{enumer}[thm]{}

\catcode`\@=11
\def\swappedhead#1#2#3{
  \thmnumber{\@upn{\@secnumfont\hspace*{-0.25em}#2\@ifnotempty{#1}{~~~}}}%
  \thmname{#1}%
  \thmnote{{\the\thm@notefont(#3)}}}
\def\endmathdisplay@a{%
  \if@eqnsw \gdef\df@tag{\theequation} \fi
  \if@fleqn \@xp\endmathdisplay@fleqn
  \else \ifx\df@tag\@empty \else \veqno \alt@tag \df@tag \fi
    \ifx\df@label\@empty \else \@xp\ltx@label\@xp{\df@label}\fi
  \fi
  \ifnum\dspbrk@lvl>\m@ne
    \postdisplaypenalty -\@getpen\dspbrk@lvl
    \global\dspbrk@lvl\m@ne
  \fi
}
\renewenvironment{equation}{%
  \stepcounter{thm}
  \incr@eqnum
  \mathdisplay@push
  \st@rredfalse \global\@eqnswtrue
  \mathdisplay{equation}%
}{%
  \endmathdisplay{equation}%
  \mathdisplay@pop
  \ignorespacesafterend
}
\catcode`\@=12

\renewcommand{\theequation}{
    \arabic{section}.\arabic{thm}%
    }

\newcommand{\cl}{\text{cl}}

\newcommand{\reals}{\mathbb{R}}
\newcommand{\rationals}{\mathbb{Q}}
\newcommand{\Prob}{\mathbb{P}}
\newcommand{\Exp}{\mathbb{E}}
\newcommand{\B}{\mathcal{B}}
\newcommand{\R}{\mathcal{R}}
\newcommand{\C}{\mathcal{C}}
\newcommand{\HH}{\mathcal{H}}
\newcommand{\G}{\mathcal{G}}
\newcommand{\E}{\mathcal{E}}
\newcommand{\W}{\mathcal{W}}
\newcommand{\F}{\mathcal{F}}
\newcommand{\ds}{\displaystyle}

\let\origcup=\cup
\let\origcap=\cap
\renewcommand{\cup}{\ds \mathop{\origcup}}
\renewcommand{\cap}{\ds \mathop{\origcap}}

\makeatletter
\def\Ddots{\mathinner{\mkern1mu\raise\p@
\vbox{\kern7\p@\hbox{.}}\mkern2mu
\raise4\p@\hbox{.}\mkern2mu\raise7\p@\hbox{.}\mkern1mu}}
\makeatother

\begin{document}

\title{Constructing Strong Markov Processes}

\author{Robert J. Vanderbei}
\address{Dept. of Mathematics,
Univ. of Illinois,
Champaign-Urbana, IL.
}
\dedicatory{Dedicated to the memory of Lynda Singshinsuk.}

\date{Original draft: August 1984.  Converted to LaTeX: \today}

\thanks{Original version is archived at \\
\hspace*{0.6in}{http://orfe.princeton.edu/$\sim$rvdb/tex/StrongMarkovProcesses/CSMP.pdf}
}

\subjclass{ 60J25, 47D07 }

\keywords{strong Markov property, Ray-Knight compactification, 
	completion of state space}

\begin{abstract}
The construction presented in this paper can be briefly described as follows:
starting from any ``finite-dimensional'' Markov transition function $p_t$, on a
measurable state space $(E,\B)$, we construct a strong Markov process on a
certain ``intrinsic'' state space that is, in fact, a closed subset of a finite
dimensional Euclidean space $\reals^d$. Of course we must explain the meaning of
finite-dimensionality and intrinsity. Starting with $p_t$, we consider the range of
the nonnegative bounded measurable functions under the action of the resolvent.
This class of functions induces a uniform structure on $E$. We say that $E$ is
finite-dimensional if this uniformity is finitely generated. In such cases we
then map $E$ into $\reals^d$. The intrinsic state space is the closure of the range of
this mapping. On this enlarged state space we construct a strong Markov process,
which corresponds quite naturally to $p_t$.  We give several examples including the
usual examples of nonstrong Markov process. 
\end{abstract}

\maketitle

\section{Introduction and notation} \label{sec1}

Every student of probability learns very early that not every Markov process
possesses the strong Markov property. And, as is often the case, the simplest
pathological examples are very simple indeed. For instance, the process that
moves deterministically to the left on the real line with unit speed except that
when it reaches zero it pauses there for an exponential amount of time before
continuing, is perhaps the best such example. The fact that it is a Markov
process follows from the memoryless property of the exponential random variable.
To see that it is not a strong Markov process, consider the first hitting time
of the open left half-line. Starting from this time, the process proceeds
immediately into the left half-line. However, at this random time the process is
situated at the origin and consequently if it had the strong Markov property it
would have to remain there for an exponential amount of time. The fact that it
does not shows that it is not strong Markov. 

Of course the above pathology can be remedied by splitting $0$ into two points
$0^+$
and $0^-$. If we put the exponential alarm clock at $0^+$ and have the process
immediately appear at $0^-$ when the clock rings (i.e. a right continuous
		trajectory) then in this enlarged state space the process is
strong Markov. 

This idea of enlarging the state space first appeared in the paper \cite{Yus57} by
Yushkevich. Shortly after that, D. Ray published a fundamental paper \cite{Ray59} in
which he introduced a certain compactification of the state space on which the
Markov process becomes what is today called a Ray process (which has the strong
Markov property). Ray's methods were clarified and improved by
F. Knight \cite{Kni65}. The short monograph \cite{Get75} by Getoor has become the standard
reference on Ray processes and the Ray-Knight compactification. However, the
most general results appeared later in the papers
\cite{Eng77a,Eng77b,Eng77c,Eng77d,Eng77e,Eng77f,Eng77g} of H. J. Engelbert. His
basic assumption is that there exists a certain class of nonnegative bounded
measurable functions, on a countably generated measurable state space, that
forms a separable Ray cone (see Definition 12 in \cite{Eng77c}). 

The starting point of this paper is a Markov transition function on a measurable
state space. We introduce in a canonical way a certain class $R_+$ of functions
that satisfy most of the requisite properties of a Ray cone. The class $R_+$ is
generally not separable (even for the above example) but is usually ``finite
dimensional.'' This means that there exists a finite 
collection $E = \{ e_1,e_2,\ldots,e_d \}$ of real-valued 
functions (not necessarily bounded) such that, for every
sequence of points $x_n$ in the state space $E$, $f(x_n)$ is Cauchy for every 
$f \in R_+$
if and only if $e_j(x_n)$ is Cauchy for $i = 1,2, \ldots ,d$. 
By thinking of the functions $e_j$ as coordinate
functions, the map $x \rightarrow (e_1(x),e_2(x), \ldots ,e_d(x))$ 
embeds the state space in $\reals^d$.
We then take the closure of the range of this mapping to get our intrinsic state
space. Thinking in terms of our original state space E, this closure corresponds
to a completion of the state space relative to the uniformity generated by $R_+$
(for a general discussion of uniformization see e.g. \cite{Kel65} Chapter VI). On the
resulting closed space we construct a strong Markov process that corresponds
naturally to the original Markov transition function. 

The advantage of our approach is that the construction is canonical. The 
Ray-Knight compactification, on the other hand, depends on the separable Ray cone
with which one starts. In fact a substantial part of the theory of the
Ray-Knight compactification consists in studying to what extent the new state
space depends on the Ray cone (see e.g. \cite{Get75}, Chapter 15). 

The main goal of this paper is the construction of a process with the strong
Markov property. Along the way we get some other results; e.g. right continuity
and the existence of left limits 
(recently this has been called cadlaguity---a gauche term indeed). 
Since the state space we construct is a
closed subset of $\reals^d$, the further sample path properties are for the most part
immediate consequences of already well known results and so they will not be
pursued here. 

The remainder of this section is devoted to establishing notations that will be
used throughout the rest of the paper. 
\begin{defn}
A {\em kernel} $k$ on a measurable space $(E,\B)$ is a real-valued function of 
$E \times \B$ satisfying: 
\begin{enumer}
\label{1.1a} $x \rightarrow k(x,B)$ is $\B$-measurable; 
\end{enumer}
\begin{enumer}
\label{1.1b} $B \rightarrow k(x,B)$ is a signed measure having finite total variation measure
$|k|(x,\cdot)$. 
\end{enumer}
\end{defn}
The formula 
\begin{equation} \label{1.3}
	\| k \| = \sup_x |k|(x,E) 
\end{equation}
defines a norm on the space $K$ of all kernels for which \eqref{1.3} is finite. It is
easy to check that $K$ is a Banach space. The formula 
\[
	k * k'(x,B) = \int k(x,dy)k'(y,B) 
\]
makes $K$ into an algebra. We denote by $k^{(n)}$ the $n^{\text{th}}$ power 
of the kernel $k$. We also put 
\[
	k^{(0)}(x,B) = 1(x,B) = 
	\left\{
	\begin{array}{ll}
		1 & x \in B  \\
		0 & x \not \in B. 
	\end{array}
	\right.
\]
It is easy to see that $\|k * k'\| \le \|k\| \|k'\|$ and hence that 
$\|k^{(n)}\| \le \|k\|^n$. 
A kernel $k$ is called a {\em stochastic kernel} if, for every $x \in E$, 
$k(x,\cdot)$ is a probability measure. 

If $(E,\B)$ is a measurable space, we also denote by $\B$ the Banach space of bounded
$\B$-measurable functions on $E$ and we denote by $\B^+$ those $f \in \B$ that are
nonnegative. The norm on $\B$ is given by $\|f\| = \sup_{x \in E} |f(x)|$.
For every kernel (say $k$) on $(E,\B)$, we denote by the corresponding upper case
letter ($K$ in this case) the linear operator on $\B$ defined by 
\[
	Kf(x) = \int k(x,dy)f(y). 
\]
If $B \rightarrow k(x,B)$ is a positive measure then $K$ is a positive operator 
and if $k$ is a stochastic kernel then $K1 = 1$ and $\|Kf\| \le \|f\|$. 

If $E$ is a subset of $\reals^d$, we denote by $C^k(E)$ the collection of all bounded
continuous functions on $E$ that have limits on the closure of $E$ and that
are $k$
times continuously differentiable with each derivative being bounded. We will
write $C(E)$ instead of $C^0(E)$. The notation $f \in C^k(E_1,E_2, \ldots ,E_n)$ 
indicates that, for each $j$, the function $f$ restricted to $E_j$ is of class 
$C^k(E_j)$. The
notation $C_0(E)$ denotes those functions in $C(E)$ that tend to zero at infinity.
Hence if $E$ is bounded, $C_0(E) = C(E)$. Finally $\overline{E}$ denotes the closure
of $E$. 

\section{The intrinsic state space} \label{sec2}

\begin{defn}
A family of kernels $(p_t)_{t>0}$ on a measurable space $(E,\B)$ is called a
{\em Markov transition function} if: 
\begin{enumer}
\label{2.1} $(t,x) \rightarrow p_t(x,B)$ is $\B(\reals^+) \times \B$ measurable; 
\end{enumer}
\begin{enumer}
\label{2.2} every $p_t$ is a stochastic kernel; 
\end{enumer}
\begin{enumer}
\label{2.3} $p_{s+t} = p_s * p_t$. 
\end{enumer}
\end{defn}

The space $(E,\B)$ is called the {\em state space} for $p_t$. 

\begin{defn}
A family of kernels $(u^{\alpha})_{\alpha>0}$ on a measurable space $(E,\B)$ is 
called a {\em Markov resolvent} if 
\begin{enumer}
\label{2.4} for every $\alpha > 0$, $\frac{1}{\alpha} u^{\alpha}$ is a stochastic kernel; 
\end{enumer}
\begin{enumer}
\label{2.5} $u^{\beta} - u^{\alpha} = (\alpha - \beta) u^{\alpha} * u^{\beta}$.
\end{enumer}
\end{defn}

If $p_t$ is a Markov transition function, the formula 
\[
    u^{\alpha} = \int_0^{\infty} e^{- \alpha t} p_t dt
\]
defines a Markov resolvent. 

For the rest of this section, let $p_t$ denote a fixed
Markov transition function on a fixed state space $(E,\B)$ and let $u^{\alpha}$ be the
corresponding resolvent. We put 
\[
    \C = U^{\alpha} \B .
\]
To see that the right hand side does not depend on $\alpha$ (as long as 
$\alpha > 0$) rewrite property \eqref{2.5} as 
$u^{\beta} = u^{\alpha} * [1 + (\alpha-\beta)u^{\beta}]$. 
Hence, if $f = U^{\beta}g$ and $g \in \B$, then 
$f = U^{\alpha}h$ where
$h = g + (\alpha - \beta)U^{\beta}g$ clearly also belongs to $\B$. Put 
\[
    \C^+ = \mathop{\cup}_{\alpha > 0} U^{\alpha} \B^+ . 
\]
The following properties of $\C$ and $\C^+$ are easily verified: 

\begin{enumer}
\label{2.6} $U^{\alpha} \C^+ \subset \C^+$, $U^{\alpha} \C \subset \C$; 
\end{enumer}
\begin{enumer}
\label{2.7} $\C^+$ is a convex cone, $\C$ is a linear space; 
\end{enumer}
\begin{enumer}
\label{2.8} $\C^+$ contains the nonnegative constants; 
\end{enumer}
\begin{enumer}
\label{2.9} $\C = \C^+ - \C^+$. 
\end{enumer}

\begin{defn}
A function $f$ is called {\em $\alpha$-supermedian} if $f \in \B^+$ and 
$f \ge (\beta - \alpha) U^{\beta} f$ for all $\beta > 0$. 
\end{defn}

\begin{prop}
\label{2.10} The following properties hold:  \\
(i) every $f \in \C^+$ is $\alpha$-supermedian for some $\alpha > 0$; \\
(ii) for every $f \in \C$, $\lim_{t \downarrow 0} \| P_t f -f \| = 0$. 
\end{prop}

\begin{proof}
Property (i) is an immediate consequence of the definition of $\C^+$ and the
resolvent property \eqref{2.5}. To prove (ii), put $g = U^{\alpha}f$. It is not hard to show
that 
\[
    e^{-\alpha t} P_t g - g = \int_0^t e^{- \alpha s} P_s f ds.
\]
Hence 
\[
    0 \le \| e^{-\alpha t} P_t g - g \| \le t \| f \|,
\]
and so 
\begin{eqnarray*}
    0 & \le & \| P_t g - g \| \\
      & \le & \| P_t g - e^{-\alpha t} P_t g \| + \| e^{-\alpha t} P_t g - g \| \\
      & \le & (1 - e^{-\alpha t}) \| g \| + t \| f \| .
\end{eqnarray*}
Letting t tend to zero establishes (ii). 
\end{proof}

It is easy to see that the minimum of two $\alpha$-supermedian functions is again
$\alpha$-supermedian. Also convex combinations of $\alpha$-supermedian functions are again
$\alpha$-supermedian.

\begin{thm}
\label{2.11} There exists a unique minimal convex cone $\R_+$ such that:  
\begin{enumerate}
\item $\C^+ \subset \R_+ \subset \B^+$; 
\item for $\alpha > 0$, $U^{\alpha}\R_+ \subset \R_+$; 
\item $f,g \in \R_+$ implies that $f \wedge g \in \R_+$. 
\end{enumerate}
In addition, every $f \in \R_+$ is $\alpha$-supermedian for some $\alpha > 0$. 
\end{thm}

\begin{proof}
If $\HH$ is any convex cone in $\B^+$, we define the usual enlargements
$U(\HH)$ and $\Lambda(\HH)$ as follows: 
\begin{eqnarray*}
    U(\HH)       & = & \{ U^{\alpha_1} f_1 + \cdots + U^{\alpha_n} f_n \; | \;
			    \alpha_j > 0, \; f_j \in \HH, \; 1 \le j \le n, \; n \ge 1 \} \\ 
    \Lambda(\HH) & = & \{ f_1 \wedge f_2 \wedge \cdots \wedge f_n \; | \;
    			    f_j \in \HH, \; 1 \le j \le n, \; n \ge 1 \}.
\end{eqnarray*}
It is easy to check that $U(\HH)$ and $\Lambda(\HH)$ are again convex cones in $\B^+$. 
Also the sum of two convex cones is again a convex cone. 

Let $\R_0 = \C^+$ and define $\R_n$ recursively by 
$\R_n = \Lambda(\R_{n-1} + U(\R_{n-1}))$.
Put $\R_+ = \cup_n \R_n$.  Note that $\R_n \subset \R_{n+1}$ for each $n \ge 0$. 
By induction on $n$, it follows that each 
$\R_n$ is a convex cone in $\B^+$ and that every $f \in \R_n$ is $\alpha$-supermedian 
for some $a > 0$. 
Hence $\R_+$ also has these properties. If $f,g \in \R_+$ then there exists an 
$n \ge 0$ with $f,g \in \R_n$.  Then 
$f \wedge g \in \R_{n+1} \subset \R_+$. 
Also, if $f \in \R_n$ then $U^{\alpha}f \in \R_{n+1} \subset \R_+$. It is obvious from the
construction that $\R_+$ is the unique minimal convex cone that satisfies
(i)-(iii). 
\end{proof}

\begin{defnumb}
\label{2.12} Two classes $\HH$ and $\G$ of real-valued functions defined on a set $E$
are said to generate the same uniformity on $E$ if, for every sequence $x_n \in
E$, the following are equivalent: 
\begin{enumerate}
\item $h(x_n)$ is a Cauchy sequence for every $h \in \HH$,
\item $g(x_n)$ is a Cauchy sequence for every $g \in \G$. 
\end{enumerate}
\end{defnumb}

We write $\HH \sim \G$ to indicate that $\HH$ and $\G$ generate the same uniformity. We say that
{\em the uniformity generated} by a class $\G$ is {\em finitely generated} if there exists a
finite collection $\HH$ such that $\HH \sim \G$. We will only consider finitely generated
uniformities (which is why we do not need to introduce nets). 

The basic example to keep in mind is the following: let $E = \reals^d$, $\HH$ be the set of
coordinate functions and $\G$ be the set of bounded continuous functions. Then 
$\HH \sim \G$
and since $\HH$ is finite, $\G$ is finitely generated. In this case $h(x_n)$ is Cauchy
for all $h \in \HH$ if and only if $x_n$ is a Cauchy sequence relative to the Euclidean
metric. 

It is easy to see that the relation $\sim$ is transitive: $\HH \sim \HH'$ and
$\HH' \sim \HH"$ implies that $\HH \sim \HH"$. We also have 

\begin{prop}
\label{2.13} If $\HH \sim \G$, then $\HH$ separates points $x$ and $y$ if and
only if $\G$ does. 
\end{prop}

\begin{proof}
If $h(x) = h(y)$ for all $h \in \HH$ then $(h(x),h(y),h(x),h(y), \ldots)$ is a Cauchy
sequence for every $h \in \HH$. Hence $(g(x),g(y),g(x),g(y), \ldots )$ is a Cauchy sequence
for every $g \in \G$ and so $g(x) = g(y)$ for all $g \in \G$. 
\end{proof}

Put $\R = \R_+ -\R_+$. We are now ready to introduce our basic assumption on $p_t$: 

\begin{assumption}
\label{2.14} There exists a finite collection $\E = \{e_1,e_2,\ldots,e_d\}$ of real-valued
functions on $E$ 
(not necessarily bounded) such that $\E \sim \R$. 
\end{assumption}

The examples in Section 4 show that this condition is often readily verifiable.
In analogy with the example given above, it is good to think of functions in $\R$ as
bounded continuous functions and those in $\E$ as coordinate functions. With this in mind
we see that \ref{2.14} is the assumption that the state space is finite dimensional. 

In assumption \ref{2.14}, it is clear that we could replace $\R$ by $\R_+$. Define a map
$\psi$ from $E$ into $\reals^d$ by 
\[
   \psi(x) = \left( e_1(x), e_2(x), \ldots, e_d(x) \right) . 
\]
The mapping $\psi$ may fail to be injective (see Example \ref{4.8}). In any case, we will
abuse language and refer to $\psi$ as an embedding of $E$ into $\reals^d$. Even
if $\psi$ is not injective, we have the following result. 

\begin{prop}
\label{2.15} If $\psi(x) = \psi(y)$ then $f(x) = f(y)$ for all $f \in \R$. 
\end{prop}

\begin{proof}
Immediate consequence of Proposition \ref{2.13}. 
\end{proof}

It follows that $\psi$ is injective if and only if $\R$ separates points. 

\begin{prop}
\label{2.16} Suppose that $\E = \{ e_1, e_2,\ldots,e_d \}$ and 
$\E' = \{ e_1' ,e_2', \ldots ,e_{d'}' \}$ satisfy 
\ref{2.14} and let $\psi$ and $\psi'$ be the corresponding embeddings. Then there exists a
continuous bijection between $\overline{\psi(\E)} \subset \reals^d$ and
$\overline{\psi'(\E)} \subset \reals^{d'}$ whose inverse is also continuous. 
\end{prop}

\begin{proof}
Since $\sim$ is transitive we see that $\E \sim \E'$. Hence, it follows from
proposition \ref{2.13} that 
$\psi' \circ \psi^{-1}$ and $\psi \circ (\psi')^{-1}$ are well defined. It is
also clear that $x_n$ is a Cauchy sequence in $\psi(\E)$ if and only if
$x_n' = \psi' \circ \psi^{-1}(x_n)$ is Cauchy in $\psi'(\E)$. Hence 
$\psi' \circ \psi^{-1}$ and $\psi \circ (\psi')^{-1}$ 
can both be extended to the closures of their
domains and give us the desired bijections. 
\end{proof}

Proposition \ref{2.16} shows that our construction is canonical and that the specific
choice of $\E$ is just a matter of choosing coordinate functions. 

Proposition \ref{2.15} shows that the following formula defines unambiguously a
mapping $\Psi$, from $\R$ into the space of real-valued functions on $\psi(\E)$: 
\[
	(\Psi f)(\xi) = f(\psi^{-1}(\xi)), \quad \xi \in \psi(\E), \; f \in \R. 
\]
The mapping $\Psi$ is linear and injective. Also note that $\Psi e_i$ makes sense and is
the $i^{\text{th}}$ coordinate function on $\psi(\E)$. 

\begin{prop}
\label{2.17} For every $f \in \R$, $\Psi f$ is a continuous function on
$\psi(\E)$. 
\end{prop}

\begin{proof}
Suppose that $\xi_m \rightarrow \xi$. Choose $x_m$, and $x$ in $E$ 
such that $x_m \in \psi^{-1}(\xi_m)$ and $x \in \psi^{-1}(\xi)$. 
Then $(e_i(x_1),e_i(x),e_i(x_2),e_i(x),\ldots )$ is a Cauchy sequence for each
$i$. By assumption \ref{2.14}, $( f(x_1), f(x), f(x_2), f(x), \ldots)$ is a Cauchy 
sequence for every $f \in \R$. 
Hence $f(x_m) \rightarrow f(x)$ for every $f \in \R$ and so 
$\Psi f(\xi_m) \rightarrow \Psi f(\xi)$ for every $f \in \R$. 
\end{proof}

For each $f \in \R$ we can extend the definition $\Psi f$ to $\overline{\psi(\E)}$ 
by continuity. To see this, suppose that $\xi_m \rightarrow \xi$ and 
$\xi_m \in \psi(\E)$.  Then $e_i(x_m)$ is Cauchy for each $i = 1, \ldots ,n$
where $x_m$ is chosen from $\psi^{-1}(\xi_m)$. Assumption \ref{2.14} then implies
that $f(x_m)$ is Cauchy for every $f \in \R$. Denote this new function by
$\overline{\Psi f}$. Put $\R' = \{ \overline{\Psi f} \; : \; f \in \R \}$, 
$\R_+' = \{ \overline{\Psi f} \; : \; f \in \R_+ \}$ and $E' = \overline{\Psi(\E)}$. 
The space $E'$ is called the
{\em intrinsic state space} for $p_t$.  The class $\R_+'$ is a convex 
subcone of $\C^+ (E')$ that
contains the nonnegative constant functions and is closed under pointwise
minimization. Also $\R' = \R_+' - \R_+'$ so it follows that $\R'$ is a linear subspace of
$C(E')$ which contains the constant functions and is closed under pointwise
minimization and maximization. We will require one final property of $\R'$ which we
state as a Lemma. 

\begin{lemma}
\label{2.18} The class $\R'$ separates points in $E'$. 
\end{lemma}

\begin{proof}
Suppose that $\xi, \eta \in E'$ and $\overline{\Psi f}(\xi) = \overline{\Psi f}(\eta)$ 
for all $f \in \R$.  Let $\xi_n$ be a
sequence of points in $\psi(\E)$ that converges to $\xi$ and let $\eta_n$ be a sequence
converging to $\eta$. Then 
$\lim_n \Psi f (\xi_n) = \lim_n \Psi f (\eta_n)$.
Let $x_n$ be a point in
$\psi^{-1}(\xi_n)$
and let $y_n$ be a point in 
$\psi^{-1}(\eta_n)$.
Then we have $\lim_n f(x_n) = \lim f(y_n)$. Hence, the sequence
$(f(x_1),f(y_1),f(x_2),f(y_2),\ldots)$
is Cauchy for every $f \in \R$.  By assumption \ref{2.14}, 
$(e_i(x_1), e_i(y_1), e_i(x_2), e_i(y_2), \ldots)$ is Cauchy for
$i = 1, \ldots,d$. Hence $\lim_n e_i(x_n) = \lim_n e_i(y_n)$ for all $i$. 
This means that $\lim \xi_n = \lim \eta_n$ and so $\xi = \eta$.
\end{proof}

For each $\alpha > 0$, define a positive linear operator $V^{\alpha}$ on $\R'$ as follows: 
\[
    V^{\alpha} \; : \; 
    \overline{\Psi f} 
    \rightarrow \Psi f 
    \rightarrow f 
    \rightarrow U^{\alpha} f
    \rightarrow \Psi(U^{\alpha} f)
    \rightarrow \overline{\Psi(U^{\alpha} f)}.
\]
The first map is simply the restriction map. The second map is $\Psi^{-1}$, which
exists since $\Psi$ is an injection. The fourth and fifth maps make sense since
$U^{\alpha} \R \subset \R$. It is easy to see that $V^{\alpha}$ inherits the 
following properties from $U^{\alpha}$: 
\begin{enumerate}
\item $\alpha V^{\alpha} 1 = 1$; 
\item $V^{\alpha} - V^{\beta} = (\beta - \alpha)V^{\alpha}V^{\beta}$; 
\item if $f_n \downarrow 0$ pointwise and $f_n \in \R'$, 
	then $V^{\alpha} f_n \downarrow 0$ pointwise. 
\end{enumerate}

\noindent
By Daniel's theorem (see e.g. \cite{DM78} Section III-35), for each $x \in E'$, and
$\alpha > 0$, there exists a unique measure $\nu^{\alpha}(x)$ on $\sigma(\R')$ 
such that 
\[
    V^{\alpha} f(x) = \int \nu^{\alpha}(x,dy) f(y), \quad f \in \R' .
\]
Since $\R'$ separates points in $E'$ and $E'$ is closed, it follows from
Proposition \ref{A.1}
that $\sigma(\R') = \B(E')$. From the lattice version of the monotone class theorem (see
e.g. \cite{DM78}, Theorem 22.3) and the fact that $x \rightarrow \nu^{\alpha}(x,f)$ is
$\B(E')$
measurable (even continuous) for every $f \in \R'$, it follows that 
$x \rightarrow \nu^{\alpha}(x,f)$ is
$\B(E')$ measurable for every bounded $\B(E')$ measurable $f$. 	
Another monotone class argument shows that $\nu^{\alpha}$ is a resolvent. 

In the next section, we start with a state space $E$, a family of real-valued functions
$\R_+$ on $E$, and a resolvent $u^{\alpha}$ which all together satisfy the conditions which we
have found that $E'$, $\R_+'$ and $\nu^{\alpha}$ satisfy. From these we show that such a resolvent
gives rise to a unique 
Markov transition function and that from this transition function we can
construct a strong Markov process in the space of right continuous trajectories
that have left hand limits. 

\section{From resolvent to strong markov process} \label{sec3}

In this section, we take as our starting point properties of the intrinsic state
space that were developed in Section \ref{sec2}. Hence we make the following
assumptions. Let $(E,\B)$ be a closed subset of $\reals^d$ with the Borel
$\sigma$-algebra.
Suppose that there is given a convex subcone $\R_+$ of $C^+(E)$ that contains the
nonnegative constant functions and that separates points in $E$. Finally, suppose
that we have a Markov resolvent $U^{\alpha}$ on $(E ,\B)$ with the following
properties: 
\begin{enumer}
\label{3.1} $U^{\alpha} \R_+ \subset \R_+$;
\end{enumer}
\begin{enumer}
\label{3.2} every $f \in \R_+$ is $\alpha$-supermedian for some $\alpha > 0$. 
\end{enumer}
\noindent
Put $\R = \R_+ - \R_+$.  We need one additional assumption that is not a carry-over from
Section 2 (or maybe it is---this is an open problem): 
\begin{enumer}
\label{3.3} $C_0(E) \subset \cl(\R)$. 
\end{enumer}

The notation $\cl(\R)$ means the closure of $\R$ in the sup norm. If $E$ is a bounded
subset of $\reals^d$ then $C_0(E) = C(E)$ and, by the lattice version of the
Stone-Weierstrass theorem, $C(E) = \cl(\R)$. Hence, in this case assumption
\eqref{3.3} is
automatic. The examples in Section 4 show that, even when $E$ is unbounded,
assumption \eqref{3.3} is easily verified. A resolvent that satisfies
\eqref{3.1}
and \eqref{3.2} will be called an {\em Engelbert resolvent}. 

The assumptions in force in this section are very similar to those of a Ray
resolvent. Indeed for a {\em Markov Ray resolvent} $U^{\alpha}$ one assumes 
that $E$ is a compact
metric space, $\B$ is the Borel $\sigma$-algebra, and 
\begin{enumer}
\label{3.4} $U^{\alpha}C(E) \subset C(E)$; 
\end{enumer}
\begin{enumer}
\label{3.5} the continuous $\alpha$-supermedian functions separate points in $E$. 
\end{enumer}

We have the following connection between Ray and Engelbert resolvents. 

\begin{thm}
\label{3.6} If $E$ is a compact subset of $\reals^d $ and $U^{\alpha}$ is a Markov Ray resolvent,
	then $U^{\alpha}$ is an Engelbert resolvent. 
\end{thm}

\begin{proof}
We must find a family $\R_+$ with the desired properties. We may take
$\R_+$ to
be the family of all continuous functions that are $\alpha$-supermedian for
some $\alpha > 0$. For any $f \in \B^+$, $U^{\alpha}f$ is $\alpha$-supermedian 
and, by \eqref{3.4}, $U^{\alpha}$ maps continuous
functions into continuous functions. Hence $U^{\alpha} \R_+ \subset \R_+$. 
Property \eqref{3.5} says
that $\R_+$ separates points. It is clear that $\R_+$ is a convex subcone of
$C^+(E)$ and so the proof is complete. 
\end{proof}

With the aim of constructing a strong Markov process, our first goal is to
invert the Laplace transform to recover $p_t$ from $U^{\alpha}$.

\begin{thm}
\label{3.7} Let $U^{\alpha}$ be an Engelbert resolvent. Then there exists a unique
Markov transition function $p_t$ such that the following hold: 
\begin{enumerate}
\item  For all $\alpha > 0$, $f \in \B$: 
	$U^{\alpha}f = \int_0^{\infty} e^{-\alpha t} P_t f dt$.
\item For all $f \in \R_+$, there exists $\alpha > 0$ such that  $e^{-\alpha t} P_t f \le f$. 
\item For all $x \in E$, $f \in \R$, the mapping $t \rightarrow P_t f(x)$ 
	is right continuous on $[0,\infty)$. 
\end{enumerate}
\end{thm}

\begin{proof}
We will need several general properties of resolvents and $\alpha$-supermedian
functions. For convenience, we have summarized these properties in the appendix.
Fix $f \in \R_+$. Then there is an $\alpha > 0$ such that $f$ is $\alpha$-supermedian. 
According to \ref{A.3} the 
$\lim_{\beta \rightarrow \infty} (\beta - \alpha) U^{\beta} f(x)$
exists. We denote the limit by $\hat{f}(x)$. Fix $x \in E$ 
and put 
$g(\beta) = \hat{f}(x)-(\beta-\alpha)U^{\beta}f(x)$. 
By \ref{A.3}, $g$ is nonnegative and by \eqref{A.2}
\[
    (-1)^n \left( \frac{d}{d \beta} \right)^n g(\beta)
    = n! (U^{\beta})^n [ 1 - (\beta-\alpha) U^{\beta} ] f .
\]
Since $f$ is $\alpha$-supermedian and $U^{\beta}$ is a positive operator, we see
that $g$ is completely monotone. 
Also $g(0) = g(0+) = \hat{f}(x) + \alpha U^0f(x)$, which might be infinite
($U^0 f(x) = \lim_{\beta \downarrow 0} U^{\beta}f(x)$ 
 exists for every $f \in \B^+$). By the Hausdorf-Bernstein-Widder
Theorem 
(see e.g. \cite{Fel71}, p.439) there exists a positive measure
$\lambda_{x,\alpha}(f,\cdot)$ on $[0,\infty)$ of
total mass $\hat{f}(x) + \alpha U^0 f(x)$ such that, for $\beta \ge 0$, 
\begin{equation} \label{3.8}
    \hat{f}(x) - (\beta - \alpha) U^{\beta} f(x) 
	= \int_{[0,\infty)} e^{-\beta t} \lambda_{x,\alpha} (f,dt) .
\end{equation}
Note that $\lambda_{x,\alpha} (f,\cdot)$ does not charge $\{0\}$: 
\begin{eqnarray*}
    \lambda_{x,\alpha}(f,\{ 0 \} ) 
    & = & \lim_{\beta \rightarrow \infty} 
    	  \int_{[0,\infty)} e^{-\beta t} \lambda_{x,\alpha} (f,dt) \\
    & = & \lim_{\beta \rightarrow \infty} 
    	  \left[ \hat{f}(x) - (\beta - \alpha) U^{\beta} f(x) \right] \\
    & = & 0.
\end{eqnarray*}
Put, for $f \in \R_+$,
\begin{equation} \label{3.9}
    P_t^{(\alpha)} f(x) 
	= e^{\alpha t} \int_{(t,\infty)} e^{-\alpha s} \lambda_{x,\alpha} (f,ds).
\end{equation}
Then for each $x \in E$, the mapping
$t \rightarrow e^{-\alpha t} P_t^{(\alpha)} f(x)$
is decreasing, right continuous and  
\[
    e^{-\alpha t} P_t^{(\alpha)} f (x)
    \; \uparrow \; 
    \int_{(0,\infty)} e^{-\alpha t} \lambda_{x,\alpha}(f,ds)
    \; = \;
    \hat{f}(x) 
\]
as $t \downarrow 0$, since 
$\lambda_{x,\alpha}(f,\cdot)$
does not charge $\{0\}$.  Moreover, for $\beta > 0$,  
\begin{eqnarray*}
    \int_0^{\infty} e^{-\beta t} P_t^{(\alpha)} f(x) dt
    & = & \int_0^{\infty} e^{-\beta t} e^{\alpha t} 
    	  \int_t^{\infty} e^{-\alpha s} \lambda_{x,\alpha}(f,ds)dt \\
    & = & \int_0^{\infty} \int_0^s e^{-(\beta-\alpha)t} dt \; 
			e^{-\alpha s} \lambda_{x,\alpha} (f,ds) \\
    & = & \frac{1}{\beta-\alpha} \int_0^{\infty} 
    	  \left[ 1 - e^{-(\beta-\alpha)s} \right] 
	  e^{-\alpha s} \lambda_{x,\alpha} (f,ds) \\
    & = & \frac{1}{\beta-\alpha} 
    	  \left[ \hat{f}(x) - \hat{f}(x) + (\beta-\alpha) U^{\beta} f(x) \right] \\
    & = & U^{\beta} f(x). 
\end{eqnarray*}
Hence $P_t^{(\alpha)}$ does not depend on $\alpha$ 
(by the uniqueness of the Laplace transform) and so we will now denote it simply
by $P_t$. 

Consider $f_1, f_2, \ldots, f_n \in \R_+$ and nonnegative real numbers 
$r_1,r_2, \ldots ,r_n$. Put $f = r_1 f_1 + \cdots + r_n f_n$. 
Choose $\alpha$ sufficiently large so that $f_1, f_2, \ldots, f_n$
are all $\alpha$-supermedian. Then $f$ is also $\alpha$-supermedian. Since 
\[
    \hat{f}(x) - (\beta-\alpha) U^{\beta} f(x) 
	= \sum_{i=1}^n r_i \left[ \hat{f}_i - (\beta-\alpha)U^{\beta}f_i(x)\right], 
	\qquad \beta > 0,
\]
it follows from the uniqueness of Laplace transforms that 
$\lambda_{x,\alpha}(f, \cdot) = \sum_{i=1}^n r_i \lambda_{x,\alpha}(f_i,\cdot)$.
It then 
follows from \eqref{3.9} that $P_t f(x) = \sum_{i=1}^n r_i P_t f_i(x)$. 
Hence $P_t$ is a cone map on $\R_+$.
Extend by 
linearity the definition of $P_t$ to $\R = \R_+ -\R_+$. Denote the extension
again by $P_t$.
It is easy to see that this extension is well defined and that 
$f \rightarrow P_t f(x)$ is a linear functional on $\R$. 
Moreover, by linearity, $t \rightarrow P_t f(x)$ is right continuous and 
\[
U^{\beta} f(x) = \int_0^{\infty} e^{-\beta t} P_t f(x) dt,
	\qquad f \in \R. 
\]
Using the first property in \eqref{A.2}, we see that 
\begin{eqnarray*}
    \left( \frac{d}{d \beta} \right)^n 
    \left[ \frac{f}{\beta} - U^{\beta+\alpha} f \right]
    & = & \frac{(-1)^n n!}{\beta^{n+1}} f - n!(-1)^n (U^{\beta+\alpha})^{n+1} f \\
    & = & \frac{(-1)^n n!}{\beta^{n+1}} n! 
    		\left[ f - (\beta U^{\beta+\alpha})^n f \right].
\end{eqnarray*}
For $f \in \R_+$, choose $\alpha$ so that $f$ is $\alpha$-supermedian. Then we can iterate the
inequality $\beta U^{\beta+\alpha} f \le f$
to conclude that $\frac{f}{\beta} - U^{\beta+\alpha} f$ is completely monotone. Since
it is clearly the Laplace transform of 
$f - e^{-\alpha t} P_t f$,
we see that $e^{-\alpha t} P_t f \le f$ for all $t \ge 0$.

Since $\R$ contains the constant functions and $U^{\alpha} 1 = 1/\alpha$, 
we see that $P_t 1 = 1$ for all $t \ge 0$.

If $f \in \R$ and $f$ is nonnegative, then 
$\left( \frac{d}{d \beta} \right)^n U^{\beta} f = (-1)^n n! (U^{\beta})^{n+1} f$
and since $(U^{\beta})^{n+1} f \ge 0$
we see that $\beta \rightarrow U^{\beta} f$ is completely monotone. Since
$U^\beta f$
is the Laplace transform of $P_t f$, we see that $P_t f \ge 0$ for all $t \ge 0$. 
Hence $P_t$ is a positive linear operator on $\R$ and, since $P_t 1 = 1$, 
it is bounded and therefore continuous. We can now extend the 
definition of $P_t$ by continuity to $\cl(\R)$. Fix $t$ and $x$ and consider the linear
functional $f \rightarrow P_t f(x)$ defined on $C_0(E)$. Let $f_n$ be a sequence of nonnegative
functions in $C_0(E)$ that decreases to zero pointwise. Then by Dini's lemma
$f_n$ converges uniformly to zero. Hence $P_t f_n(x)$ converges to zero and so we can
apply Daniel's theorem to conclude that there exists a unique measure
$p_t(x,\cdot)$ on $\B$ such that $P_t(x) = \int p_t(x,dy)f(y)$ for all 
$f \in C_0(E)$. 

By Proposition A.4, the map $(t,x) \rightarrow P_t f(x)$ is 
$\B([0,\infty)) \times \B$ measurable for each $f \in \R$. 
Applying the lattice version of the monotone class theorem,
it follows that $(t,x) \rightarrow P_t f(x)$ is $\B([0,\infty)) \times \B$ measurable for
every $f \in \B$ and also that 
\[
    U^{\alpha} f(x) = \int_0^{\infty} e^{-\alpha t} P_t f(x) dt
\]
for every $f \in \B$. 

Finally, we must show that $P_t$ forms a semigroup. Again by a monotone class
argument, it suffices to show that 
$P_t P_s f = P_{t+s} f$
for all $f \in \R_+$, and $s,t > 0$. Fix
$f \in \R_+$. Then $P_t P_s f$ and $P_{t+s} f$ are both right continuous in 
$s$ and so it suffices
to show that, for each $t \ge 0$, they have the same Laplace transform on $s$. That
is, 
\[
    P_t U^{\alpha} f(x) = \int_0^{\infty} e^{-\alpha s} P_{t+s} f(x) ds.
\]
Since $U^{\alpha} f \in \R$, the left-hand side is right continuous in $t$. So is the right-hand
side. Hence, it suffices to check that they have the same Laplace transform on
$t$. Thus 
\begin{eqnarray*}
    U^{\beta} U^{\alpha} f(t) 
    & = & \int_0^{\infty} 
          e^{-\beta t} \int_0^{\infty} e^{-\alpha t} P_{t+s} f(x) ds dt \\
    & = & \int_0^{\infty} \int_0^{\infty}
          e^{-(\beta -\alpha)t} e^{-\alpha s} P_s f(x) ds dt \\
    & = & \int_0^{\infty} e^{-\alpha s} 
          \left [ \int_0^{\infty} e^{-(\beta -\alpha)t} dt \right] P_s f(x) ds .
\end{eqnarray*}
Since 
\[
    \int_0^s e^{-(\beta - \alpha)t} dt 
    = \left\{ 
         \begin{array}{ll}
	     \frac{1}{\beta-\alpha} \left[ 1 - e^{-(\beta-\alpha)s} \right] 
	     & \qquad \beta \ne \alpha \\
	     s & \qquad \beta = \alpha 
	 \end{array}
      \right.
\]
we must check that 
\[
    U^{\beta} U^{\alpha} f
    = \left\{ 
         \begin{array}{ll}
	     \frac{1}{\beta-\alpha} \left[ U^{\alpha} f - U^{\beta} f \right] 
	     & \qquad \beta \ne \alpha \\
	     (U^{\alpha})^2 f & \qquad \beta = \alpha 
	 \end{array}
      \right.
\]
which is true since $U^{\alpha}$ is a resolvent. This completes the proof. 
\end{proof}

The transition function constructed in Theorem \ref{3.7} will be called an {\em Engelbert
transition function}. Let $\W$ denote the space of all functions from 
$\reals^+ = [0,\infty)$
into $E$ that are right continuous and have left-hand limits. Put 
$Y_t(w) = w(t)$ for $w \in \W$ and let 
$\G^0 = \sigma\{Y_t | t \ge 0\}$ and $\G_t^0 = \sigma\{Y_t | 0 \le s \le t \}$. 

\begin{thm} \label{3.10} 
Let $p_t$ be an Engelbert transition function. Then. for each
probability measure $\mu$ on $(E ,\B)$, there exists a probability measure
$\Prob_{\mu}$ on $(\W,\G^0)$ such that 
\begin{enumerate}
\item $\Exp_{\mu} \{ f(Y_{t+s}) | \G_s^0 \} = P_t f(Y_s)$, a.s. $\Prob_{\mu}$
for $f \in \B$ and all $s, t \ge 0$;
\item $\Exp_{\mu} f(Y_t) = \int \mu(dx) P_t f(x)$, for $f \in \B$ and all $t \ge
0$
\end{enumerate}
($\Exp_{\mu}Z$ denotes the expectation of $Z(w)$ with respect to the probability
 measure $\Prob_{\mu}$). 
\end{thm}

\begin{proof}
The proof is quite standard and we will present essentially the same
outline as in \cite{Get75}, Theorem 5.1. Let $\Omega = E^{[0,\infty)}$ and 
$\F = \B^{[0,\infty)}$. Put $X_t(\omega) = \omega(t)$ and 
$\F_t = \sigma\{X_s | 0 \le s \le t \}$. Given $\mu$, we can apply the Kolmogorov extension
theorem to obtain a measure $\Prob$ on $(\Omega,\F)$ such that
$(X_t,\F_t,\Prob)$ satisfies properties (i) and (ii). 

We must show that this process has a modification that is right continuous and
has left-hand limits. By Theorem \ref{3.7}-(ii), for each $f \in \R_+$, 
there exists an $\alpha > 0$ such that 
$e^{-\alpha t} P_t f \le f$.
This fact together with (i) implies that 
$e^{-\alpha t} f(X_t)$ is an $\F_t$-supermartingale. 
Hence, with $\Prob$-probability one, the map $t \rightarrow f(X_t)$ restricted
to the positive rationals $\rationals^+$ has left- and right-hand limits at each point in
$\reals^+$. Since $\cl(\R) \supset C_0(E)$, it follows that for each $f \in
C_0(E)$, with $\Prob$-probability
one, the map $t \rightarrow f(x_t)$ restricted to $\rationals^+$ has left and right-hand limits at each
point in $\reals^+$. The space $C_0(E)$ contains a countable subset 
$(f_n)_{n \ge 0}$ that separates the points in $E$. 
Let $N_n$ be a $\Prob$-null set off of which
$f_n(X_t)$ restricted to $\rationals^+$ has left- and right-hand limits at each
point of $\reals^+$. Put $N = \ds \cup_n N_n$. Since $(f_n)_{n \ge 0}$ separates points, 
it follows that, on $\Omega \setminus N$, the map $t \rightarrow X_t$, 
restricted to $\rationals^+$ has right- and left-hand
limits in $E$ at each point in $\reals^+$. Take $\Omega \setminus N$ as our new
sample space $\Omega$. For each $t \in \reals^+$ and $\omega \in \Omega$, put 
\[
    X_{t+}(\omega) 
	= \lim_{\stackrel{r \downarrow t}{r \in \rationals^+}} X_r(\omega) .
\]
Fix $f,g \in \R_+$, $t \in \reals^+$ and let $r_n$ be a sequence in
$\rationals^+$ such that $r_n \downarrow t$. Then, applying (ii) we get 
\begin{equation} \label{3.11}
  \begin{array}{rcl}
    \Exp f(X_t) g(X_{t+})
    & = & \ds \lim_{n \rightarrow \infty} \Exp f(X_t) g(X_{r_n}) \\[0.2in]
    & = & \ds \lim_{n \rightarrow \infty} \Exp f(X_t) P_{r_n-t}g(X_t) \\[0.2in]
    & = & \Exp f(X_t) \hat{g}(X_t) \\[0.2in]
    &  & \phantom{\Exp f(X_t) \hat{g}(X_t) } \\[0.2in]
    &  & \phantom{\Exp f(X_t) \hat{g}(X_t) } 
  \end{array}
\end{equation}

\vspace*{-0.6in}\noindent
where $\hat{g}(x) = \lim_{t \downarrow 0} P_t g(x) = P_0 g(x)$. 
Since $e^{-\alpha t}P_tg \le g$ (for some $\alpha > 0$) and $g \ge 0$, it
follows that $0 \le \hat{g} \le g$. Since $P_t g = P_t P_0 g = P_t \hat{g}$, we see that 
\[
    \Exp g(X_t) 
    	    = \int \mu(dx) P_t g(x) 
	    = \int \mu(dx) P_t \hat{g}(x) = \Exp \hat{g}(X_t).
\]
Hence $g(X_t) = \hat{g}(X_t)$ $\Prob$-a.s. and so \eqref{3.11} implies that 
\begin{equation} \label{3.12}
    \Exp f(X_t)g(X_{t+}) = \Exp f(X_t)g(X_t). 
\end{equation}
Since $\cl(\R) \supset C_0(E)$, \eqref{3.12} actually holds for all $f,g \in
C_0(E)$. A monotone class argument now shows that 
$\Exp F(X_t,X_{t+}) = \Exp F(X_t,X_t)$ for all $F \in \B \times \B$. Taking
$F(x,y) = 1_{x=y}$, we see that 
\[
   \Prob(X_t = X_{t+}) = 1. 
\]

Now define $Y_t(\omega) = X_{t+}(\omega)$ for each $t \in \reals^+$, 
$\omega \in \Omega$. Then the map $t \rightarrow Y_t(\omega)$
is right continuous on $\reals^+$ and has left limits on $(0,\infty)$. 
Put $\HH_t = \sigma(Y_t | s \le t)$. 
Since $Y_t = X_t$ $\Prob$-a.s., it follows that $(Y_t,\HH_t, \Prob)$ satisfies (i) and (ii).
Finally, we define a map $\pi$ from $\Omega$ into $\W$ by the formula
$\pi(\omega)(t) = Y_t(\omega)$.
Then the measure $\Prob_{\mu} = \Prob \circ \pi^{-1}$ has all the 
properties stated in the theorem. 
\end{proof}

Formula \eqref{3.10}-(i) is the basic form of the Markov property. We now extend this
property in the usual way. Indeed, for each $t \ge 0$ let $\theta_t$ denote the
{\em shift operator} in $\W$: 
\[
    (\theta_t w)(s) = w(t+s). 
\]
For each probability $\mu$ on $(E,\B)$ let $\G^{\mu}$ denote the completion of
$\G^0$ with respect to $\Prob_{\mu}$ and 
let $N^{\mu} = \{ B \in \G^{\mu} \; | \; \Prob_{\mu}(B) = 0 \}$.
Put $\G_t^{\mu} = \sigma \{ \G_t^{0} \ds \cup N^{\mu} \}$,
$\G_t = \cap_{\mu} \G_t^{\mu}$ and 
$\G = \ds \cup_{\mu} \G^{\mu}$.
For any bounded $\G$-measurable function $Z$, put 
\[
    \theta_t Z(w) = Z(\theta_t w). 
\]
Using a standard monotone class argument, formula \eqref{3.10}-(i) can be extended as
follows: 
\begin{equation} \label{3.13}
    \Exp_{\mu} \{ \theta_t Z | \G_t \} = \Exp_{Y_t} Z, \qquad Z \in \G, \; t \ge 0.
\end{equation}
Formula \eqref{3.13} is called the {\em Markov property}. 

A map $\tau$ from $\W$ to $[0,\infty]$ is called
a {\em stopping time} if 
$\{ w \; | \; \tau(w) \le t \} \in \G_t$ for all $t \ge 0$. 
If $\tau$ is a stopping time, we define the {\em pre-$\tau$ $\sigma$-algebra}
$\G_{\tau}$ to be the collection of all sets $A \in \G$
such that $A \cap \{ \tau \le t \} \in \G_t$ for all $t \ge 0$. 
The following theorem extends \eqref{3.13} to stopping times. 

\begin{thm} \label{3.14} 
Let $\Prob_{\mu}$ be the extension to $(\W,\G)$ of the measure constructed 
in Theorem \ref{3.10}.  Then, for each stopping time $\tau$ and for each 
$Z \in \G$, 
\begin{equation} \label{3.15}
    \Exp_{\mu} \{ \theta_{\tau} Z | \G_{\tau} \}
    = \Exp_{Y_{\tau}} Z \qquad \text{a.s. $\Prob_{\mu}$ on $\{ \tau < \infty \}$}.
\end{equation}
Also $\G_t = \cap_{s > t} \G_s$.
\end{thm}

\begin{proof}
The proof follows immediately from the results proved in \cite{BG68} Sections 1.6--1.8. 
We only remark that in Theorem 1.8.11 of \cite{BG68} we must replace his $\mathbb L$ by our
class $\R$. Then in his proof, the monotone class theorem can still be used since
$\sigma(\R) = \B$. 
\end{proof}

Formula \eqref{3.15} is the {\em strong Markov property}. 

\section{Examples}

\subsection{Uniform Motion.} \label{4.1} 
Let $(E,\B) = (\reals,\B(\reals))$ and let $p_t$ be the transition function 
for uniform motion: 
\[
    p_t(x,B) = 1_B(x-t). 
\]
Then 
\[
    P_t f(x) = f(x-t)  
\]
and
\[
    U^{\alpha}f(x) = e^{-\alpha x} \int_{-\infty}^{x} e^{\alpha y} f(y)dy. 
\]
It is clear that, for every $f \in \B$, $U^{\alpha}f \in C(\reals)$. Since
$C(\reals)$ is a linear lattice, it follows that $\R \subset C(\reals)$. 

Now suppose that $g \in C^1(\reals)$ and put $f = g' + \alpha g$. Then, a simple calculation shows
that $U^{\alpha}f = g$. 
Hence, $C^1(\reals) \subset \C \subset \R \subset C(\reals)$. 
The classes $C^1(\reals)$ and $C(\reals)$ generate
the same uniformity and this uniformity is also generated by $\E = \{ e(x) = x \}$.
Hence the intrinsic state space coincides with the real line. 

\subsection{Brownian Motion.} \label{4.2} 
Let $(E,\B) = (\reals^d, \B(\reals^d))$ and let $p_t$ be the transition function for 
$d$-dimensional Brownian motion: 
\[
    p_t(x,B)= \int p_t(y-x)1_B(y)dy 
\]
where 
\[
    p_t(y) = (2 \pi t)^{-d/2} \exp\{ -|y|^2/2t \} .
\]
Then 
\[
    P_t f(x) = \int p_t(y-x)f(y)dy. 
\]
It is easy to check that for every $\epsilon > 0$ there exists a $K < \infty$ such that: 
\[
    | p_t(y-x) - p_t(y-x_0) | \le K p_{2t}(y-x_0),
\]
whenever $| x - x_0 | < \epsilon$. Hence, the dominated convergence theorem shows that for
every 
$f \in \B$, $P_t f$ belongs to $C(\reals^d)$. Put $g = U^{\alpha} f$. 
By Proposition \ref{2.10}-(i), $\lim_{t \downarrow 0} \| P_t g -g \| = 0$. 
Since $x \rightarrow P_t g(x)$ is continuous, it follows that $g$ belongs to
$C(\reals^d)$.
Thus, $\C \subset C(\reals^d)$ and since $C(\reals^d)$ is a linear lattice we have 
$\C \subset \R \subset  C(\reals^d)$. 

Let $g$ be a function of class $C^2(\reals^d)$ and let 
$f = -\frac{1}{2} \Delta g + \alpha g$ where 
$\Delta = \frac{\partial^2}{\partial x_1^2} + \cdots + \frac{\partial^2}{\partial x_d^2}$
is the Laplacian in $\reals^d$. If we integrate by parts (in $x$), then
use the fact 
that 
$\frac{1}{2} \Delta p_t = \frac{d}{dt} p_t$,
and integrate by parts again (in $t$), we get 
\[
    \int_s^{\infty} e^{-\alpha t}P_t f(x) dt = e^{-\alpha s} \int p_s (y-x) g(y) dy. 
\]
Letting $s \downarrow 0$ and using the continuity of $g$ we get $U^{\alpha} f = g$. 
Hence, $C^2(\reals^d) \subset \C \subset \R \subset C(\reals^d)$. 
Since the classes $C^2(\reals^d)$ and $C(\reals^d)$ generate the same uniformity on
$\reals^d$ we see that we may take 
\[
    \E = \{ e_1(x) = x_1, \; e_2(x) = x_2, \; \ldots, \; e_d(x) = x_d \} .
\]
Hence the intrinsic state space coincides with $\reals^d$. 

\subsection{Pure Jump Process.} \label{4.3} 
Let $(E,\B)$ be a general state space, let $q$ be a stochastic kernel on $(E, \B)$
and put 
\[
    p_t(x,B) 
	= \sum_{n=0}^{\infty} e^{-\lambda t} \frac{(\lambda t)^n}{n!} q^{(n)}(x,B) .
\]
Then 
\[
    P_t f(x) = 
      \sum_{n=0}^{\infty} e^{-\lambda t} \frac{(\lambda t)^n}{n!} Q^n f(x) 
\]
and
\[
    U^{\alpha} f(x) =
      \frac{1}{\alpha+\lambda} \sum_{n=0}^{\infty} 
	\left( \frac{\lambda}{\alpha+\lambda} \right)^n Q^n f(x)
\]
where $Q$ is the operator on $\B$ corresponding to $q$. 

Suppose that $g \in \B$ and put $f = \lambda (1-Q) g + \alpha g$. 
Then a simple calculation shows that $U^{\alpha}f = g$. 
Hence $\B \subset \C$ and so it follows that $\B = \C = \R$. The class $\B$ generates a
finite dimensional uniformity if and only if $E$ is a countable set. In this case,
if points are measurable, then we can take $E = \{ e(x_n) = n \}$ which maps $E$
into the set of nonnegative integers. 

\subsection{Uniform Motion with Sticky Origin.} \label{4.4} 
Let $(E,\B) = (\reals, \B(\reals))$ and let $p_t$ be the transition function for 
uniform motion with ``stickum'' at zero (discussed in the Introduction): 
\[
    p_t(x,B) =
	\left\{
	    \begin{array}{ll}
	        \ds e^{x-t} 1_B(0) + \int_0^{t-x} e^{-u} 1_B(x-t+u) du & 
		    \qquad 0 \le x < t \\[0.2in]
		\ds 1_B(x-t) & \qquad \text{else.}
	    \end{array}
	\right.
\]
Then 
\[
    U^{\alpha}f(x) =
	\left\{
	    \begin{array}{ll}
	        \ds \frac{e^{-\alpha x}}{1+\alpha} 
		\left[ f(0) + \int_{-\infty}^0 e^{\alpha t} f(t) dt
		    +(1+\alpha) \int_0^x e^{\alpha t} f(t) dt \right] 
		    & \qquad x \ge 0 \\[0.2in]
		\ds e^{-\alpha x} \int_{-\infty}^x e^{\alpha t} f(t) dt
		    & \qquad x < 0.
	    \end{array}
	\right.
\]
It is easy to see from this formula that if $f \in \B$ then $U^{\alpha} f$ belongs to
$C((-\infty,0),[0,\infty))$. Since this last class is a linear lattice, it
follows that $\R \subset C((-\infty,0),[0,\infty))$.

Now suppose that $g \in C^1((-\infty,0),[0,\infty))$ and $g'(0) = 0$ 
(this is the right-hand derivative at zero). 
Put $f = g' + \alpha g$. Then a simple integration by parts shows
that $U^{\alpha} f = g$. Hence,
\[
    C^1((-\infty,0),[0,\infty)) \cap \{ h \; | \; h'(0) = 0 \}
    \; \subset \; \C \; \subset \; \R \; \subset \;
    C((-\infty,0),[0,\infty)) 
\]
where by $h'(0)$ we mean the right-hand derivative at zero. It is easy to see
that we can take $\E$ consisting of one function 
\[
    e(x) = 
        \left\{
	    \begin{array}{cl}
	        x    & \qquad x \ge 0, \\
	        x -1 & \qquad x  <  0.
	    \end{array}
	\right.
\]
This function makes small negative numbers ``far away'' from small positive
numbers. Also the space $\psi(E) = (-\infty,-1) \cup [0, \infty)$ is not 
closed and so the intrinsic state space 
$E' = (-\infty,-1] \cup [0, \infty)$ 
turns out to have one point more than $E$. 

\subsection{Uniform Motion with a Jump.} \label{4.5}
Let $E = (-\infty,-1] \cup [0,\infty) = E^- \cup E^+$ and let $\B$ be the Borel
$\sigma$-algebra on $E$.  Let $p_t$, be ``severed'' uniform motion: 
\[
    p_t(x,B) =
        \left\{
	    \begin{array}{ll}
	        1_B(x-t-1) & \qquad 0 \le x < t, \\
	        1_B(x-t)   & \qquad \text{else}.
	    \end{array}
	\right.
\]
Then 
\[
    P_t f(x) = 
        \left\{
	    \begin{array}{ll}
	        f(x-t-1) & \qquad 0 \le x < t, \\
	        f(x-t)   & \qquad \text{else},
	    \end{array}
	\right.
\]
and 
\[
    U^{\alpha} f(x) =
        \left\{
	    \begin{array}{ll}
	        \ds e^{-\alpha x} 
		\left[ 
		   e^{\alpha} \int_{-\infty}^{-1} e^{\alpha y} f(y) dy
		   + 
		   \int_{0}^{x} e^{\alpha y} f(y) dy
		\right]
		    & \qquad x \in E^+, \\[0.2in]
		\ds e^{-\alpha x} \int_{-\infty}^{x} e^{\alpha y} f(y) dy
		    & \qquad x \in E^- .
	    \end{array}
	\right.
\]
Note that $U^{\alpha} f(0) = U^{\alpha} f(-1)$. 
Put $I_{-1,0} = \{ h \; | \; h(-1) = h(0) \}$. 
We see that if $f \in \B$ then $U^{\alpha}$ belongs to $C(E) \cap I_{-1,0}$. 

Now suppose that $g \in C^1(E) \cap I_{-1,0}$ and put $f = g' + \alpha g$. 
Again, integrating by parts shows immediately that $g = U^{\alpha} f$. 
Hence 
\[
    C^1(E) \cap I_{-1,0} \; \subset \; \C \; \subset \; \R \; \subset \; C(E) \cap I_{-1,0} .
\]
In this example, $\R$ does not separate $-1$ and $0$. We may choose for $\E$ the
collection consisting of the single function 
\[
    e(x) = 
        \left\{
	    \begin{array}{ll}
	        x   & \qquad x \in E^+ \\
	        x+1 & \qquad x \in E^- .
	    \end{array}
	\right.
\]
On the intrinsic state space the process is just uniform motion. 

\subsection{Uniform Motion with Fork in the Road} \label{4.6}
In $\reals^2$, let $R$ denote the closed right half-line on the $x$-axis, let
$U$ denote
the open upper half-line on the $y$-axis, and let $L$ denote the open lower
half-line on the $y$-axis. Let $E = R \cup V \cup L$ and let $\B$ be the Borel
$\sigma$-algebra on $E$.  Let $p_t$ be the 
transition function for the process that proceeds deterministically left, up,
or down on $R$, $U$, or $L$, respectively. When it arrives at the origin, it proceeds
upward/downward each with probability $1/2$. Then 
\[
    P_t f(x,y) =
        \left\{
	    \begin{array}{ll}
	        f(0,y+t) & \qquad (x,y) \in U, \\
	        f(0,y-t) & \qquad (x,y) \in L, \\
	        f(x-t,0) & \qquad (x,y) \in R, \; 0 < t \le x, \\
	        \frac{1}{2} f(0,t-x) + \frac{1}{2} f(0,x-t) 
		  & \qquad (x,y) \in R, \; 0 \le x < t, \\
	    \end{array}
	\right.
\]
and
\[
    U^{\alpha} f(x,y) =
        \left\{
	    \begin{array}{ll}
		\ds e^{\alpha y} \int_{y}^{\infty} e^{-\alpha t} f(0,t) dt
			& \; (x,y) \in U , \\[0.2in]
		\ds e^{-\alpha y} \int_{-\infty}^{y} e^{\alpha t} f(0,t) dt
			& \; (x,y) \in L , \\[0.2in]
		\ds e^{-\alpha x} \int_{0}^{x} e^{\alpha t} f(t,0) dt \\[0.2in]
		  \ds \quad + \frac{1}{2} e^{-\alpha x} \int_0^{\infty} e^{\alpha t}
		    \left[ f(0,t) + f(0,-t) \right] dt
			& \; (x,y) \in R . 
	    \end{array}
	\right.
\]
From this formula for $U^{\alpha}$, we see that 
$\C \; \subset \; \R \; \subset \; C(R,U,L)$. 

Now suppose that $g \in C^1(R,U,L)$ and 
$g(0,0) = \frac{1}{2} \left[g(0,0+) + g(0,0-)\right]$. Put 
\[
    f = 
      \left\{
	\begin{array}{rl}
	  \ds -\frac{\partial g}{\partial x_2} + \alpha g   & \quad (x,y) \in U \\[0.2in]
	   \ds \frac{\partial g}{\partial x_2} + \alpha g   & \quad (x,y) \in L \\[0.2in]
	   \ds \frac{\partial g}{\partial x_1} + \alpha g   & \quad (x,y) \in R .
	\end{array}
      \right.
\]
Then, as in the previous examples, simple calculations show
that $g = U^{\alpha} f$.  Hence $g \in \C$. Since the collection of all such
functions $g$ generates the same uniformity as $C(R,U,L)$ does, 
we see that we can take $E = \{ e_1(x,y), e_2(x,y) \}$ where 
\begin{eqnarray*}
    e_1(x,y) & = & 1_R(x,y) \\
    e_2(x,y) & = &
    	\left\{
	    \begin{array}{cl}
	        y+1 & \quad (x,y) \in U \\
	        y-1 & \quad (x,y) \in L \\
	          0 & \quad (x,y) \in R.
	    \end{array}
	\right.
\end{eqnarray*}
Hence $\psi(E) = \{ (x,y) \; | \; x \ge 0, y = 0 
                      \text{ or } x = 0, y > 1 
		      \text{ or } x = 0, y < -1 \}$ 
and so the
closure $\E^1 = \overline{\psi(E)}$ contains two new points: $(0,1)$ and $(0,-1)$. Since the
intrinsic strong Markov process is right continuous we see that, with
probability one, it visits one of these two new points and it never visits
the branch point $(0,0)$. 

\subsection{Brownian Motion with an Absorbing State.} \label{4.7} 
Let $(E,\B) = (\reals^d, \B(\reals^d))$ and let $p_t$ be the transition function for
$d$-dimensional Brownian motion except it is altered at $x = 0$ as follows: 
\[
    p_t(x,B) =
        \left\{
	    \begin{array}{ll}
	        \int p_t(y-x) 1_B(y) dy & \quad x \ne 0 \\
	        1_B(0)                  & \quad x  =  0 
	    \end{array}
	\right.
\]
where $p_t(y)$ is as defined in Subsection \ref{4.2}. Since the origin is a set of measure
zero, it is easy to check that this $p_t$ is a Markov transition function. It is
also easy to see that $\R$ coincides with the $\R$ for Brownian motion except that the
value of each function at zero can be changed arbitrarily. Hence we may take 
\[
    \E = \{ 
    e_1(x) = x_1, \; 
    e_2(x) = x_2, \; \ldots, \;
    e_d(x) = x_d, \; 
    e_{d+1}(x) = 1_0(x)
    \} .
\]
Hence the intrinsic state space is a closed subset of $\reals^{d+1}$ that consists of
two parts: the $d$-dimensional hyperplane corresponding to $x_{d+1} = 0$ and the
single point corresponding to $x_{d+1} = 1$ with the rest of the coordinates
vanishing. On the hyperplane the transition function is that for ordinary
Brownian motion. The single point is an absorbing point. 

\subsection{Instantaneous Jump to Limiting Distribution.} \label{4.8} 
Let $(E,\B)$ be a general measurable space and let $p_t(x,B) = \pi(B)$ where
$\pi$ is a probability measure on $(E,\B)$. Then 
\[
    P_t f(x) = \int f(y) \pi(dy) 
\]
and
\[
    U^{\alpha} f(x) = \frac{1}{\alpha} \int f(y) \pi(dy) .
\]
Hence, for $f \in \B$, $U^{\alpha}f$ is a constant function. Hence, both $\C$
and $\R$ coincide with the collection of all constant functions. 
In this example, every point gets
identified as the same. We may take $\E$ consisting of one function 
\[
    e(x) = 1. 
\] 
In this case, the entire state space collapses to a single point.

\section{Appendix}

In this section, we collect certain results that are either well-known or
secondary to the main theme of the paper. 

\begin{prop} \label{A.1}
Let $E$ be a closed subset of $\reals^d$ and let $\HH$ be a class of
continuous real-valued functions defined on $E$. Then the $\sigma$-algebra generated by
$\HH$ coincides with the Borel $\sigma$-algebra on $E$ if and only if $\HH$
separates points in $E$. 
\end{prop}

\begin{proof}
The ``only if'' direction is trivial. Suppose that $\HH$ separates points in $E$.
For $a \in \reals^d$ and $r > 0$, let $B_r(a)$ denote the intersection of $E$ and the closed
ball of radius $r$ centered at $a$. Fix $a \in E$ and $0 < r_0 < r_1$. Put 
$A_n = \{ x \in E \; | \; n r_1 \le |x-a| \le (n + 1)r_1 \}$. 
Fix $x \in B_{r_0}(a)$ and $y \in B_{r_0}(a)^c$. Let $f$ be a function in $\HH$ for
which $f(x) \ne f(y)$. Suppose that $f(x) < f(y)$.
Pick $s$ such that $f(x) < s < f(y)$ and let $V_{x,y} = f^{-1}((-\infty,s))$ and 
$W_{x,y} = f^{-1}((-\infty,s])$.  Then, in the relative topology, $V_{x,y}$ is
open, $W_{x,y}$ is closed, $x \in V_{x,y} \subset W_{x,y}$ and 
$y \in W_{x,y}^c$.  The collection $\{ V_{x,y} \; | \; x \in B_{r_0}(a) \}$
forms an open cover of the compact set $B_{r_0}$ and hence there is a finite
subcover $V_{x_1,y}, \ldots, V_{x_n,y}$.  Put $F_y = \cup_{i=1}^n W_{x_i,y}$.
Then $F_y$ is a closed set, $y \in F_y^c$ and $B_{r_0}(a) \subset F_y$.  Also,
$F_y \in \sigma(\HH)$.

Now fix $n \ge 1$. The collection $\{ F_y^c \; | \; y \in A_n \}$ forms an open 
cover of the compact set $A_n$. 
Hence there is a finite subcover $\{ F_{y_{j,n}}^c \; | \; j = 1,2,\ldots,k_n \}$. 
The collection 
$\{ F_{y_{j,n}}^c \; | \; j = 1,2,\ldots,k_n; \; n = 1,2,\ldots \}$
is a countable collection that covers $\{ x \; | \; |x-a| \ge r_1 \}$.
Hence,
\[
    \ds
    B_{r_0}(a) \; \subset \; \cap_{j,n} F_{y_{j,n}} \; \subset \; B_{r_1}(a)
\]
and $\cap_{j,n} F_{y_{j,n}}$ is evidently $\sigma(\HH)$-measurable and closed.
Since $r_0$ and $r_1$ were arbitrary, we see that for any $0 < r_0 < r_1$ we can
find a closed $\sigma(\HH)$-measurable set $G$ such that 
$B_{r_0}(a) \subset G \subset B_{r_1}(a)$.

Now fix $r > 0$ and put $B = B_r(a)$ and $B_n = B_{r+1/n}(a)$.  There exist
closed $\sigma(\HH)$ measurable sets $G_n$ such that 
$B \subset G_n \subset B_n$.
Hence $B \subset \cap_n G_n \subset \cap_n B_n = B$ and so
$\ds B = \ds \cap_n G_n$ is a $\sigma(\HH)$-measurable set.  Since $a$ and $r$ were
arbitrary, we see that 
$\{ B_r(a) \; | \; a \in E, \; r > 0 \} \subset \sigma(\HH)$
and since these sets generate the Borel $\sigma$-algebra it follows that 
$B(E) \subset \sigma(\HH)$.   Since every function in $\HH$ is continuous, we
see that $\sigma(\HH) \subset \B(E)$.
\end{proof}

The following properties follow directly from the resolvent identity 
\eqref{2.5}. 

\begin{prop} \label{A.2}
Let $U^{\alpha}$ be a resolvent on $(E,\B)$. Then for each $f \in \B$ and 
$x \in E$, the map $\alpha \rightarrow U^{\alpha} f(x)$ 
is infinitely differentiable on $(0,\infty)$ and 
\[
    \left( \frac{d}{d \alpha} \right)^n U^{\alpha} f 
    	= (-1)^n n! (U^{\alpha})^{n+1} f 
\]
and 
\[
    \left( \frac{d}{d \alpha} \right)^n ( \alpha U^{\alpha} f )
    	= (-1)^{n+1} n! (U^{\alpha})^n \left[ 1 - \alpha U^{\alpha} \right] f .
\]
\end{prop}

\begin{proof}
The proof can be found in Getoor \cite{Get75} page 7. The next two propositions are also
proved in \cite{Get75} on pages 5 and 12, respectively. 
\end{proof}

\begin{prop} \label{A.3}
If $f$ is $\beta$-supermedian, then the map 
$\alpha \rightarrow (\alpha-\beta)U^{\alpha}f(x)$ is increasing for 
$\alpha > \beta$ and the limit
$\widehat{f} = \lim_{\alpha \rightarrow \infty} \alpha U^{\alpha}f$ 
is $\beta$-supermedian. Also $\widehat{f} \le f$ 
and $U^{\alpha} \widehat{f} = U^{\alpha} f$ for all $\alpha > 0$.
\end{prop}

%
%
%
%
%
%
%
%
%

\vfill
\pagebreak
\bibliographystyle{plain}
\bibliography{../lib/refs}


\end{document}